\documentclass[11pt]{amsart}
\usepackage{amsmath}
\usepackage{latexsym}
\usepackage{amssymb,amsthm,amsfonts}
\usepackage{hyperref}

\newtheorem{defi}{Definition}[section]
\newtheorem{theorem}[defi]{Theorem}
\newtheorem{lemma}[defi]{Lemma}

\newtheorem{prop}[defi]{Proposition}

\def\R{\mathbb{R}}

\def\C{\mathbb{C}}

\def\e{\varepsilon}
\def\D{\Delta}

\def\g{\gamma}

\def\wra{\rightharpoonup}
\def\intr{\int_{\R^N}}
\def\dis{\displaystyle}

\begin{document}
\title[A modified nonlinear Schr{\"o}dinger equation]
{Existence of ground states for a modified nonlinear Schr{\"o}dinger equation}

\author{David Ruiz}
\author{Gaetano Siciliano}

\address{Departamento de An{\'a}lisis Matem{\'a}tico, University of Granada. 18071 Granada (Spain).}

\thanks{D.R. has been supported by the Spanish
Ministry of Science and Innovation under Grant MTM2008-00988 and
by J. Andaluc\'{\i}a (FQM 116).}
\thanks{G.S. has been supported by the University of Bari under grant DR 13800, by M.I.U.R. - P.R.I.N.
``Metodi variazionali e topologici nello studio di fenomeni non lineari'' and
by J. Andaluc\'{\i}a (FQM 116).}

\email{daruiz@ugr.es, sicilia@ugr.es}


 \maketitle

\begin{abstract}
In this paper we prove existence of ground state solutions of the
modified nonlinear Schr{\"o}dinger equation:
$$ -\D u+V(x)u-\frac{1}{2}u \D u^{2}=|u|^{p-1}u, \ x \in \R^N,\ N \geq 3,
$$ under some hypotheses on $V(x)$. This model has been proposed in
the theory of superfluid films in plasma physics. As a main
novelty with respect to some previous results, we are able to deal
with exponents $p\in(1,3)$. The proof is accomplished by
minimization under a convenient constraint.

\end{abstract}

\section{Introduction}

Let us consider the following modified version of the nonlinear
Schr{\"o}dinger equation:
\begin{equation*}
i\phi_t -\D \phi+W(x)\phi-\frac{1}{2}\D
g(|\phi|^{2})g'(|\phi|^2)\phi=f(x,\phi), \ \ x \in \R^N,
\end{equation*}
where $N \geq 3$, $\phi: \R \times \R^N \to \C$ and $f(x, \phi)$
is a nonlinear term. This quasilinear version of the nonlinear
Schr{\"o}dinger equation arises in several models of different
physical phenomena, such as in the study of superfluid films in
plasma physics, in condensed matter theory, etc. (see \cite{7, 8,
15, 16, 17, 25, 27}).

In this paper we restrict ourselves to the case $g(\phi)=\phi$ and
a power nonlinearity $f(x,\phi)$:
\begin{equation}\label{Eq1}
i\phi_t -\D \phi +W(x)\phi -\frac{1}{2}\phi \D
|\phi|^{2}=|\phi|^{p-1}\phi, \ \ x \in \R^N.
\end{equation}
This equation has been introduced in \cite{5,6,14} to study a
model of a self-trapped electrons in quadratic or hexagonal
lattices (see also \cite{brihaye, brihaye2}). In those references
numerical and analytical results have been given.

From a mathematical point of view, the local existence for the
Cauchy problem has been first considered in \cite{lange, popp},
being later improved in \cite{jeanjean}. See also \cite{kenig} for
a result concerning very general quasilinear Schr{\"o}dinger
equations. In \cite{jeanjean} the orbital stability of stationary
solutions is studied, including blow-up, an issue also considered
in \cite{guo}.

Here we are interested in the existence of standing waves. By
taking $\phi=e^{-i \omega t} u(x)$ with $u: \R^N \to \R$, we are
led to the equation:

\begin{equation}\label{Eq}
-\D u+V(x)u-\frac{1}{2}u \D u^{2}=|u|^{p-1}u,
\end{equation}
where $V(x)=W(x)+\omega$.

Let us define $X=\{u \in H^1(\R^N):\ \intr u^2 |\nabla u|^2
<+\infty\} = \{u\in H^1(\R^N): \ u^2 \in H^1(\R^N)\}$. Then, $u
\in X$ is a weak solution of \eqref{Eq} if: \begin{equation}
\label{weakf} \intr (1+u^2)\nabla u \cdot \nabla \psi + u|\nabla
u|^2 \psi + V(x) u \psi -|u|^{p-1}u\psi=0 \ \ \forall \ \psi \in
C_0^{\infty}(\R^N).
\end{equation}
Observe that this is a weak formulation of a quasilinear problem
with principal part in divergence form. In a certain sense, weak
solutions are critical points of the functional $I: X \to \R$
defined by:
\begin{equation}\label{I}
I(u)=\frac{1}{2} \Big (\int_{\R^{N}} |\nabla u|^2+V(x)u^2+u^2
|\nabla u|^2 \Big )-\frac{1}{p+1}\int_{\R^{N}} |u|^{p+1}.
\end{equation}

By using Sobolev embedding and interpolation we conclude that $X
\subset L^q(\R^N)$ for any $q \in [2, \frac{4N}{N-2}]$. Indeed,
the exponent $p= \frac{4N}{N-2}-1 = \frac{3N+2}{N-2}$ is critical
with respect to the existence of solutions, see \cite{LWW2}, for
instance.

It seems quite clear that $X$ is the right space for studying the
functional \eqref{I}. However, the term $\intr u^2 |\nabla u|^2$
is not convex and $X$ is not even a vector space. So, the usual
min-max techniques cannot be directly applied to $I$.

In literature several papers have considered this problem. Maybe
the first one is \cite{PSW}, where both the one dimensional and
radial cases are studied. In \cite{LW} the higher dimension case
is considered. In both papers the proofs are based on a
minimization technique on the set:
$$ \left \{u\in X: \ \intr |u|^{p+1} = \mu \right \}. $$
Therefore, a Lagrange multiplier appears in the equation.
Precisely, existence of solutions is provided for a sequence of
different multipliers.

In \cite{colin, LWW}, through a convenient change of variables,
\eqref{Eq} is transformed into a related semilinear problem. That
semilinear problem is more tractable since one can apply the
well-known arguments of \cite{blions}. So doing, solutions are
found if either $p\geq 3$ or $V$ is constant.

Finally, in \cite{LWW2} the authors use a minimization on a
Nehari-type constraint to get existence results. Their argument
does not depend on any change of variables, so it can be applied
to more general problems. Moreover, they also provide existence of
sign-changing solutions, which seems to be a delicate issue in
this kind of problems. But again $p\geq 3$ is assumed. See also
\cite{alves, LSW, LW2, moameni, moameni2, moameni3, o, severo,
severo2} for related results.

So, there are only a few partial results in the case $p\in (1,3)$.
As previously mentioned, in \cite{LW}, \cite{PSW} a Lagrange
multiplier appears in the equation. In \cite{colin} an existence
result is given but only for constant potentials $V$ (so one can
consider the radially symmetric case). Moreover, a specific change
of variables is used.

In this paper we give an existence result for nonconstant $V$ and
$p \in (1,\frac{3N+2}{N-2})$. Moreover we avoid the use of any
change of variables, so our techniques could be of use in more
general situations. The main result of the paper is the following:

\begin{theorem} \label{teo} Assume $p\in(1, \frac{3N+2}{N-2})$ and $V \in
C^1(\R^N)$ satisfying the following:

\begin{enumerate}
    \item[(V1)] $0<V_0 \le V(x)\le \lim_{|x|\rightarrow+\infty}V(x)=V_{\infty}<+\infty,$
    \item[(V2)] the function $x \mapsto x \cdot \nabla V(x)$
    belongs to $L^{\infty}(\R^N)$,
    \item[(V3)] the map $s \mapsto s^{\frac{N+2}{N+p+1}}V(s^{\frac{1}{N+p+1}}x)$ is concave for any $x\in\mathbb R^N.$
\end{enumerate}

Then, there exists $u\in X$ a positive solution of \eqref{Eq}.
Moreover $u$ is a ground state, that is, its energy $I(u)$ is
minimal among the set of nontrivial solutions of \eqref{Eq}.

\end{theorem}

Let us comment briefly the conditions assumed on $V(x)$.
Hypothesis (V1) is quite usual in this kind of problems (other
typical assumptions are periodicity, or $\lim_{|x|\to +\infty}
V(x)=+\infty$). This condition will allow us, by using the
concentration-compactness principle of Lions (\cite{L1}), to deal
with the lack of compactness due to the effect of translations in
$\R^N$.

Condition (V2) is technical but not very restrictive since,
whenever $\lim_{|x|\to +\infty} x \cdot \nabla V(x)$ exists, it
must be zero (by (V1)). Here we can also handle bounded
oscillations at infinity of the function $x \mapsto x \cdot \nabla
V(x)$.

More restrictive is the concavity hypothesis (V3). It is used at a
unique technical but essential point, in the proof of Lemma
\ref{lemma11}. Obviously, constant functions satisfy (V3). Let us
give a nontrivial example where (V3) is verified. Take $W\in
C^2(\R^N)$ such that:

\begin{enumerate}
    \item[(W1)] $ W(x)\le \lim_{|x|\rightarrow+\infty}W(x)=W_{\infty}<+\infty$,
    \item[(W2)] both functions $x \mapsto x \cdot \nabla W(x)$, $x \mapsto D^2W(x)[x,x]$
    belong to $L^{\infty}(\R^N)$.
\end{enumerate}

As above, condition (W2) is not very restrictive. If (W1) holds
and the limits $\lim_{|x|\to +\infty} x \cdot \nabla W(x) $, $
\lim_{|x|\to +\infty} D^2W(x)[x,x]$ exist, they must be zero and
then (W2) is satisfied.

Then, one can easily check that there exists $\omega_0 \in \R$
such that for any $\omega>\omega_0$, $V(x)=W(x)+\omega$ satisfies
(V1), (V2) and (V3). Moreover $\omega_0$ can be made explicit upon
the $L^{\infty}$ norms of $W$ and the functions given in
assumption (W2). Therefore we provide existence of standing waves
for \eqref{Eq1} under conditions (W1) and (W2) (for phases
$\omega>\omega_0$).

The proof is based, again, on a constrained minimization
procedure. But here the constraint is not of Nehari-type; instead,
we use a Pohozaev identity. In general, this technique of
minimizing a functional under a Pohozaev constraint has been
rarely used in literature. As far as we know, it appears for the
first time in \cite{shatah} in the study of stationary solutions
of the nonlinear Klein-Gordon problem (see also \cite{JFA} for a
different application).

Once a minimizer has been found, we need to show that it is indeed
a solution. In order to prove that we follow the ideas of Lemma
2.5 of \cite{LWW2} (in turn inspired by \cite{castro}).

The paper is organized as follows. In Section 2 we establish some
preliminary results and state Theorems \ref{teo1} and \ref{teo2},
from which Theorem \ref{teo} follows. Section 3 is devoted to the
proofs of Theorems \ref{teo1} and \ref{teo2}.

\section{Preliminaries and statement of the results}

In this section we begin the study of \eqref{Eq}. After some
preliminaries, we show the minimization process that will be used
in the proofs. Finally, we state Theorems \ref{teo1} and
\ref{teo2}.

First of all, some comments are in order. In the Appendix of
\cite{LWW2} it is proved that weak solutions are bounded in
$L^{\infty}(\R^N)$. Let us point out that their arguments (based
on Moser and De Giorgi iterations) work also for $p \in (1,3)$.

As mentioned in the introduction, problem \eqref{Eq} is a
quasilinear elliptic equation with principal part in divergence
form. A density argument shows that the weak formulation
\eqref{weakf} holds also for test functions in $H^1(\R^N) \cap
L^{\infty}(\R^N)$. By \cite{lu} (see Theorems 5.2 and 6.2 in
Chapter 4) it follows that $u \in C^{1,\alpha}$. From Schauder
theory we conclude that $u \in C^{2,\alpha}$ is a classical
solution of \eqref{Eq}.

Moreover, if $u\in X$ is a solution, $u$, $Du$, $D^2u$ have an
exponential decay as $|x| \to + \infty$ (see again the Appendix in
\cite{LWW2} and take into account the previous regularity
discussion).

Finally, in \cite{LWW2} a Pohozaev identity for equation
\eqref{Eq} is mentioned. Let us make it explicit. Just by applying
Proposition 2 of \cite{pucci}, assume that $u \in X$ is a $C^2$
solution of \eqref{Eq} such that:
\begin{equation} \label{condpoho}
|\nabla u|^2+V(x)u^2+u^2 |\nabla u|^2 + |u|^{p+1}\ ,\ \
\frac{|u|}{1+|x|} \Big( |\nabla u|+ u^2 |\nabla u|  \Big) \ \in
L^1(\R^N), \end{equation} then, for any $a\in \R$, $u$ satisfies
the identity:
\begin{multline} \label{poho} 
\dis \Big( \frac{2-N}{2}+a \Big)\int_{\R^{N}}
|\nabla u|^2 +\Big(a- \frac{N}{2}\Big)\int_{\R^{N}} V(x) u^2 -
\frac 1 2 \int_{\R^{N}} \nabla V(x)\cdot x\, u^2 + \\ 
\Big(2a+\frac{2-N}{2} \Big) \intr |\nabla u|^2 u^2 +
\Big(\frac{N}{p+1}-a \Big) \int_{\R^{N}} |u|^{p+1}=0.
\end{multline}

Let us briefly check conditions \eqref{condpoho}. The first
condition being obvious for any $u\in X$, we consider the second
one. By using H{\"o}lder and Hardy inequalities:
$$ \intr \frac{|u|}{1+|x|}  |\nabla u| \leq \Big( \intr \frac{|u|^2}{|x|^2} \Big)^{1/2}
\Big( \intr |\nabla u|^2 \Big)^{1/2} \leq C \intr |\nabla
u|^2<+\infty,
$$
and,
$$ \intr \frac{u^2}{1+|x|} \ |u| \,|\nabla u| \leq \Big( \intr \frac{u^4}{|x|^2} \Big)^{1/2}
\Big( \intr u^2|\nabla u|^2 \Big)^{1/2} \leq C \intr |\nabla
u^2|^2< + \infty.
$$

Let us now say a few words on the space $X$:
\begin{equation*}
X=\{u\in H^1(\mathbb{R}^N): u^2\in H^1(\mathbb{R}^N) \}
\end{equation*}
where $H^{1}(\R^{N})$ is the usual Sobolev space. As mentioned in
the introduction, $X$ is not a vector space (it  is not closed
under the sum), nevertheless it is a complete metric space with
distance:
$$d_{X}(u,v)=\|u-v\|_{H^1}+\| \nabla u^2-\nabla v^2\|_{L^2}.$$
It is easy to check that $I$ is continuous on $X$. Moreover, for
any $\psi \in C_0^{\infty}(\R^N)$ and $u \in X$, $u+\psi \in X$,
and we can compute the Gateaux derivative:
\begin{equation*}
\langle I'(u),\psi\rangle=\intr (1+u^2)\nabla u \cdot \nabla \psi
+ u|\nabla u|^2 \psi + V(x) u \psi -|u|^{p-1}u\psi.
\end{equation*}
Therefore, $u \in X$ is a solution of \eqref{Eq} if and only if
the Gateaux derivative of $I$ along any direction in
$C_0^{\infty}(\R^N)$ vanishes.

For any $u\in X$ we hereafter denote by $u_{t}$ the map:
$$\mathbb R^+ \ni t \mapsto u_t \in X,\ u_t(x)=tu(t^{-1}x).$$ It is an exercise to check
that $t \mapsto u_t$ is indeed a continuous curve in $X$ (for
example, use Brezis-Lieb lemma \cite{brezis-lieb}). Let us compute
the functional $I$ on that curve:
\begin{multline}\label{f}
    f_u (t):=I(u_t)=\frac{t^{N}}{2}\int_{\R ^{N}} |\nabla u|^2+\frac{t^{N+2}}{2}\int_{\R^{N}} V(tx) u^2\\
    +\frac{t^{N+2}}{2}\int_{\R^{N}} |\nabla u|^2 u^2-\frac{t^{N+p+1}}{p+1}\int_{\R^{N}}
    |u|^{p+1}.
\end{multline}
Since $p+1>2$, we get that
\begin{itemize}
\item $f_u (t)>0$ for $t>0$ sufficiently small,
 \item $\lim_{t\rightarrow+\infty} f_u (t)=-\infty$.
\end{itemize}
This implies that $f_u$ attains its maximum. Moreover, thanks to
(V2), $f_u: \R^+ \to \R$ is $C^1$ and:
\begin{multline*}
f_u'(t)=\frac{N}{2}t^{N-1}\int_{\R^{N}} |\nabla
u|^2+\frac{N+2}{2}t^{N+1}\int_{\R^{N}} V(tx)\, u^2+\frac{t^{N+1}}
{2}\int_{\R^{N}} \nabla V(tx)\cdot tx\, u^2 \\
     +\frac{N+2}{2}t^{N+1}\int_{\R^{N}} |\nabla u|^2 u^2-\frac{N+p+1}{p+1} t^{N+p}\int_{\R^{N}} |u|^{p+1}. \nonumber
\end{multline*}
This motivates the following definition:
\begin{eqnarray*}
M=\left\{u\in X\setminus\{0\}: J(u)=0 \right\},
\end{eqnarray*}
where $J:X \to \R$ is defined as:
\begin{multline*}
 J(u)=\frac{N }{2}\int_{\R^{N}}
|\nabla u|^2+ \frac{N+2}{2}\int_{\R^{N}} \Big( V(x)\, u^2+|\nabla u|^2 u^2\Big) +\\
\frac{1}{2}\int_{\R^{N}} \Big (\nabla V(x)\cdot x\, u^2 \Big )
    -\frac{N+p+1}{p+1} \int_{\R^{N}} |u|^{p+1}.$$
\end{multline*}

In other words, $M$ is the set of functions $u$ such that
$f_u'(1)=0$. Moreover, for $t\neq 0, f'_{u_{t}}(1)=t f'_{u}(t)$.
Observe also that $M$ is nothing but the set of functions $u \in
X$ such that the Pohozaev identity \eqref{poho} holds for $a=-1$.
Then, all solutions belong to $M$.

\medskip

We now state Theorems \ref{teo1} and \ref{teo2}, that will be
proved in the next section.

\begin{theorem}\label{teo1} Define $m:=\inf_M I$. Then $m$ is positive and is
achieved at some $u\in M$.
\end{theorem}

\begin{theorem}\label{teo2}
The minimizer $u$ provided by Theorem \ref{teo1} is a ground state
solution of equation \eqref{Eq}. Moreover, it is positive (up to a
change of sign).
\end{theorem}

\section{Proof of the main results}
In order to prove Theorems \ref{teo1} and \ref{teo2}, we need
several auxiliary results. The proofs of the theorems will be
given in two final subsections.

\begin{lemma} \label{lemma11} For any $u \in X-\{ 0\}$, the map
$f_u$ defined in \eqref{f} attains its maximum at exactly one
point $t^{u}$. Moreover, $f_u$ is positive and increasing for
$t\in [0,t^u]$ and decreasing for $t>t^u$. Finally,
$$m=\inf_{u\in X, u\ne0} \max_{t>0} I(u_t)\,.$$
\end{lemma}
\begin{proof}

First of all, by making the change of variable $s=t^{N+p+1}$, we
get
$$f_u (s)=\frac{1}{2}s^{\frac{N}{N+p+1}}\int_{\R^{N}} |\nabla u|^2
    +\frac{s^{\frac{N+2}{N+p+1}}}{2}\int_{\R^{N}} (V(s^{\frac{1}{N+p+1}}x)u^2
    + |\nabla u|^2 u^2)-\frac{s}{p+1}\int_{\R^{N}} |u|^{p+1}.$$
By assumption (V3) (and this is the only but essential point in
which (V3) is used!) this is a concave function. We already know
that it attains its maximum; let $t^u$ be the unique point at
which this maximum is achieved. Then $t^u$ is the unique critical
point of $f_u$ and $f_u$ is positive and increasing for $0<t<t^u$
and decreasing for $t>t^u$.

In particular, for any $u\neq 0$, $t^u\in \R$ is the unique value
such that $u_{t^u}$ belongs to $M$, and $I(u_t)$ reaches a global
maximum for $t=t^u$. This finishes the proof.
\end{proof}

The first claim of Theorem \ref{teo1} is  proved here; it is
indeed a consequence of a suitable comparison argument.
\begin{lemma}
There holds $m>0.$
\end{lemma}
This readily implies that $d_X(M,0)>0$.
\begin{proof}
Let us define $$\bar{I}(u)=\frac{1}{2}\int_{\R^{N}}\left( |\nabla
u|^2+V_0\, u^2+u^2 |\nabla u|^2\right)- \frac{1}{p+1}\int_{\R^{N}}
|u|^{p+1},$$ where $V_0$ comes from $(V1)$. Obviously, $\bar{I}(u)
\leq I(u)$, and this implies that:
$$\bar{m}:= \inf_{u\in X, u\ne0} \max_{t>0} \bar{I}(u_t) \leq \inf_{u\in X, u\ne0} \max_{t>0} I(u_t)=m. $$
It suffices then to prove that $\bar{m}>0$. Let us define:
$$\bar{M}=\left\{u\in X\setminus\{0\}: g_u'(1)=0\right\} \ \ \text{where} \ \ g_u (t)=\bar{I}(u_t).$$
By Lemma \ref{lemma11} applied to $V \equiv V_0$ (actually a more
direct proof can be given in this case), we know that
$$ \bar{m} = \inf_{u \in \bar{M}} \bar{I}(u).$$
For any $u\in \bar{M}$,
\begin{align*}
\frac{N+2}{2}V_0 \int_{\R^{N}} u^2+\frac{N+2}{2}\int_{\R^{N}} |\nabla u|^2 u^2&\le\frac{N+p+1}{p+1}\int_{\R^{N}} |u|^{p+1}\\
&\le \frac{N+2}{2} V_0 \int_{\R^{N}} u^2+ C \int_{\R^{N}}
|u|^{\frac{4N}{N-2}}
\end{align*}
for a suitable constant $C>0$. So, by using the Sobolev
inequality,
$$\frac{N+2}{2}\int_{\R^{N}}|\nabla u|^{2}u^{2}dx\le C\int_{\R^{N}} |u|^{4N/N-2}\le C' \left(\int_{\R^{N}} |\nabla u|^2 u^2\right)^{\frac{N}{N-2}}$$
and this shows that $\int_{\R^{N}} |\nabla u|^2 u^2$ is bounded
away from zero on $\bar{M}$.

We conclude since the functional $\bar{I}$ restricted to $\bar{M}$
has the expression 
$$
\bar{I}(u) =\frac{1}{2}\frac{p+1}{N+p+1}\int_{\R^{N}} |\nabla
u|^2+\frac{V_0}{2}\frac{p-1}{N+p+1}\int_{\R^{N}} u^2
+\frac{p-1}{N+p+1}\int_{\R^{N}}|\nabla u|^2 u^2\,.
$$

\end{proof}

In a certain sense, the next proposition deals with the
``coercivity" of $I|_{M}$.

\begin{prop}\label{prop1} There exists $c>0$ such that for any $u\in
M$,
$$I(u) \geq c \int_{\R^N} \Big( u^2 + |\nabla u|^2 + u^2 |\nabla
u|^2\Big).$$
\end{prop}

\begin{proof}

Take $u \in M$ and choose $t\in(0,1)$. We compute:
\begin{align*}
I(u_t)-t^{N+p+1}I(u)&=\Big( \frac{t^{N}}{2} - \frac{t^{N+p+1}}{2}
\Big) \int_{\R ^{N}} |\nabla u|^2+
 \Big( \frac{t^{N+2}}{2} - \frac{t^{N+p+1}}{2} \Big) \int_{\R^{N}} |\nabla u|^2 u^2 \\
 &+ \int_{\R^{N}} \Big( \frac{t^{N+2}}{2}V(tx) - \frac{t^{N+p+1}}{2} V(x) \Big) u^2.
\end{align*}
Observe that $V(tx) \geq V_0 \geq \delta V_{\infty} \geq \delta
V(x)$, for some positive constant $\delta \in(0,1)$ depending only
on $V_0$ and $V_{\infty}$. By choosing a smaller $t$, if
necessary, we get that
$$ \frac{t^{N+2}}{2}V(tx) - \frac{t^{N+p+1}}{2} V(x) \geq
\Big(\delta \frac{t^{N+2}}{2 } - \frac{t^{N+p+1}}{2}\Big)  V(x)
\geq \gamma
$$
for a positive fixed constant $\gamma>0$. Recall that, by Lemma
\ref{lemma11}, $I(u_t) \leq I(u)$; by taking a smaller $\gamma$,
if necessary,
$$(1-t^{N+p+1})I(u) \geq I(u_t)-t^{N+p+1}I(u) \geq \gamma \int_{\R^N} \Big( u^2 + |\nabla u|^2 + u^2 |\nabla
u|^2\Big).$$ We conclude by defining $c =
\frac{\gamma}{1-t^{N+p+1}}$.
\end{proof}

\subsection{Proof of Theorem \ref{teo1}}
Take a sequence $u_n \in M$ so that $I(u_n) \to m$. By Proposition
\ref{prop1}, both $u_n$ and $u_n^2$ are bounded in $H^1(\R^N)$.
Passing to a convenient subsequence, both $u_n$ and $u_n^2$
converge weakly in $H^1(\R^N)$ and also pointwise. Therefore,
\begin{equation*}
u_{n}\wra u \ \ \text{ and } \ \ u_{n}^{2}\wra u^{2}\ \  \text{ in
} \ \  H^{1}(\R^{N}).
\end{equation*}
This implies in particular that $\{u_{n}\}$ is  bounded in
$L^{p+1}(\R^{N})$.

The proof proceeds in several steps.
\medskip

{\bf Step 1:} $\dis \intr |u_n|^{p+1} \nrightarrow 0$.

\medskip

Let us recall the expression of $I$:
$$ I(u_n)= \frac{1}{2}\int_{\R^{N}}\left( |\nabla u_n|^2+V(x)u_n^2+u_n^2 |\nabla u_n|^2\right)-\frac{1}{p+1}\int_{\R^{N}}
|u_n|^{p+1}\to m.
$$
First of all, since  $m >0$, $\|u_n\|_{H^1} + \|u_n^2\|_{H^1}
\nrightarrow 0$ by Proposition \ref{prop1}. By Lemma
\ref{lemma11}, for any $t>1$,
\begin{align*}
m \leftarrow &I(u_n)\geq I((u_n)_t) \\
&= \frac{t^{N}}{2}\int_{\R ^{N}} |\nabla u_n|^2+
\frac{t^{N+2}}{2}\int_{\R^{N}} V(tx) u_n^2 +
\frac{t^{N+2}}{2}\int_{\R^{N}} |\nabla u_n|^2
u_{n}^2-\frac{t^{N+p+1}}{p+1}\int_{\R^{N}} |u_n|^{p+1} \\
&\geq \frac{t^N}{2} \int_{\R ^{N}}\left( |\nabla u_n|^2+ V_0 u_n^2
+
  |\nabla u_n|^2 u_{n}^2 \right) - \frac{t^{N+p+1}}{p+1}\int_{\R^{N}} |u_n|^{p+1}\\
 &\geq \frac{t^N}{2} \delta  - \frac{t^{N+p+1}}{p+1}\int_{\R^{N}} |u_n|^{p+1},
\end{align*}
where $\delta$ is a fixed constant. It suffices to take $t>1$ so
that $\frac{t^N \delta}{2} >2 m$ to get a lower bound for $\intr
|u_n|^{p+1}$, proving Step 1.

So, let us assume (passing to a subsequence, if necessary), that
\begin{equation}\label{convA}
    \int_{\R^{N}} |u_n|^{p+1}\rightarrow A\in(0,\infty).
\end{equation}
\medskip

{\bf Step 2:} Splitting by concentration-compactness.

\medskip

In this step we use in an essential way the
concentration-compactness principle. 
We recall here the following result due to P. L. Lions (Lemma I.1
of \cite{L1}, part 2):
%
%
%
%
%



\begin{lemma} Let $1< r \leq \infty$, $1\leq q < \infty$ with $q \neq \frac{Nr}{N-r}$ if $r<N$.
Assume that $\psi_n$ is bounded in $L^q(\R^N)$,  $\nabla \psi_n$
is bounded in $L^r(\R^N)$ and:
$$ \sup_{y\in \R^N} \int_{B_y(R)}|\psi_n|^q \to 0\ \  \mbox{ for some }\ \  R>0. $$
Then $\psi_n \to 0$ in $L^{\alpha}(\R^N)$ for any $\alpha \in (q,
\frac{Nr}{N-r})$.
\end{lemma}

We apply the previous lemma to $\psi_n=u_n^2$, $q=\frac{p+1}{2}$
and $r=2$; by Step 1, $\psi_n$ does not converge to $0$ in
$L^{q}(\R^N)$. By interpolation, this implies that $\psi_n$ does
not converge to $0$ in $L^{s}(\R^N)$ for any $s \in (1,
\frac{2N}{N-2})$.

Then, there exists $\delta>0$ and $\{x_{n}\}\subset \R^N$ such
that
$$\int_{B_{x_n}(1)}|u_n|^{p+1} >  \delta>0.$$

Fix $\e >0 $ and take $R> \max\{ 1, \ \varepsilon^{-1}\}$,
$\eta_{R}(t)$ a smooth function defined on $[0,\, +\infty)$ such
that
\begin{itemize}
\item[a)] $\eta_{R}(t)=1$ for $0\le t\le R$, \item[b)]
$\eta_{R}(t)=0$ for $t\ge 2R$, \item[c)] $\eta'_{R}(t)\le 2/R$.
\end{itemize}
Define
\begin{equation*}
    v_n(x)=\eta_{R}(|x-x_n|)u_n (x) \ \ \ \mbox{and} \ \ \ w_n (x)=\left(1-\eta_{R}(|x-x_n|)\right)u_{n}(x).
\end{equation*}
Clearly $v_{n}$ and $w_{n}$ belong to $X$ and $u_{n}=v_{n}+w_{n}$.
Observe that in particular
\begin{equation}\label{nonvanish}
\liminf_{n\rightarrow+\infty} \int_{B_{x_{n}}(R)} |v_{n}|^{p+1}
\geq \delta.
\end{equation}

\bigskip

{\bf Step 3:} There exist constants $C>0$, $\tau>0$ independent of
$\e$ and $n_{0}=n_0(\e)$ such that $\|w_n\|_{H^1} + \|w_n^2
\|_{H^1} \leq C \e^{\tau}$ for all $n \geq n_0$.

\bigskip
%

In the rest of the proof we denote by $C$ certain positive
constants, that may change from one expression to another, but all
of them are independent of $\e$ and $n$.

\medskip

Define $z_n=u_n(\cdot+x_n)$. Clearly, $z_n \rightharpoonup z$ and
$z_n^2 \rightharpoonup z^2$ (both weak convergences are understood
in $H^1(\R^{N})$). By taking a larger $R$, if necessary, we can
assume that $\int_{A_0(R,\ 2R)} |z|^{p+1} < \e$, where $A_0(R,2R)$
denotes the anulus centered in $0$ with radii $R$ and $2R$. Then,
for $n$ large enough,
\begin{equation}\label{i}
\left | \int_{\R^{N}} |u_n|^{p+1}-\int_{\R^{N}}|v_n|^{p+1}
-\int_{\R^{N}} |w_n|^{p+1} \right | \le 3\varepsilon.
\end{equation}

\medskip Since $|\nabla z_{n}|^{2}$ is uniformly bounded in
$L^1(\R^N)$, up to a subsequence, $|\nabla z_{n}|^{2}$ converges
(in the sense of measure) to a certain positive measure $\mu$ with
$\mu(\R^{N})<+\infty$. By enlarging $R$, if necessary, we can
assume that $\mu \Big( A_0(R,2R) \Big)<\e$. Then, for $n$ large
enough,
$$\int_{\R^{N}}|\nabla u_{n}|^{2} \eta_{R}(|x-x_{n}|)(1-\eta_{R}(|x-x_{n}|))dx<
\e.$$
Taking this into account, straightforward computations show that
for $n$ large enough,
\begin{eqnarray}\label{ii}
\left | \int_{\R^{N}} |\nabla u_n|^2 -  \int_{\R^{N}} |\nabla
v_n|^2-\int_{\R^{N}}|\nabla w_n|^2 \right|& =& \left |
2\int_{\R^{N}}\nabla w_{n}\nabla v_{n} \right | \nonumber \\
&\leq& \frac{C}{R}+2\varepsilon \leq C\,\varepsilon.
\end{eqnarray}
Arguing as before, possibly choosing a larger $R$, we get also
\begin{equation}\label{iii}
\left|\int_{\R^{N}} |\nabla u_n|^2 u_n^2 - \int_{\R^{N}} |\nabla
v_n|^2 v_n^2 - \int_{\R^{N}} |\nabla w_n|^2 w_n^2 \right | \leq
C\,\varepsilon\,
\end{equation}
and finally
\begin{equation}\label{iv}
\left | \int_{\R^{N}} V(tx)u_n^2  - \int_{\R^{N}} V(tx)v_n^2 -
\int_{\R^{N}} V(tx)w_n^2 \right | \leq C \e.
\end{equation}
Putting together \eqref{i}, \eqref{ii}, \eqref{iii} and \eqref{iv}
we get that for $n$ sufficiently large and $t>0$,
\begin{equation}\label{500}
    |I((u_n)_t) - I((v_n)_t)-I((w_n)_t)| \le C\varepsilon
    (t^N+t^{N+p+1}).
\end{equation}

\bigskip

Now let us denote with $t^{v_{n}}$ and $t^{w_{n}}$ the posi\-tive
values which maximize $f_{v_{n}}(t)$ and $f_{w_{n}}(t)$
respectively, namely,
\begin{equation*}
I((v_{n})_{t^{v_{n}}})=\max_{t>0}I((v_{n})_{t}) \ \ \text{and} \ \
I((w_{n})_{t^{w_{n}}})=\max_{t>0}I((w_{n})_{t})\,.
\end{equation*}
First, let us assume that $t^{v_{n}} \leq  t^{w_{n}}$ (the other
case will be treated later). Then,
\begin{equation}\label{pos}
I((w_{n})_{t})\ge 0 \ \ \text{ for } \ \ t \leq t^{v_n}.
\end{equation}
The next aim is to find suitable bounds for the sequence
$\{t^{v_{n}}\}.$

\medskip

{\bf Claim:} There exist $0 < \tilde{t} <1<\bar{t}$ independent of
$\e$ such that $t^{v_n} \in (\tilde{t},\bar{t}\,)$.

\medskip
Indeed, take $ \bar{t}= \Big ( (p+1) (A+1)^{-1} B
\Big)^{\frac{1}{p-1}}$, where $A$ comes from \eqref{convA} and $B$
is large enough such that $\bar{t}>1$ and moreover
\begin{equation} \label{B} B \geq
\int_{\R^{N}}|\nabla u_{n}| ^{2}+\int_{\R^{N}}V_{\infty} u_{n}^{2}
+\int_{\R^{N}}|\nabla u_{n}|^{2} u_{n}^{2}.\end{equation}
 Then
 \begin{eqnarray*}
I((u_n)_{\bar{t}}) &\leq& \frac{\bar{t}^{N+2}}{2} \Big (
\int_{\R^{N}}|\nabla u_{n}| ^{2}+\int_{\R^{N}}V(\bar{t}x)
u_{n}^{2} +\int_{\R^{N}}|\nabla u_{n}|^{2} u_{n}^{2} -
\frac{2}{p+1} \bar{t\,}^{p-1} \intr |u_n|^{p+1} \Big)\\
&\leq& -B \frac{\bar{t}^{N+2}}{2}<0.
\end{eqnarray*}
By \eqref{500}
\begin{equation}\label{550}
I((u_{n})_{t})\ge I((v_{n})_{t})+I((w_{n})_{t})- C \varepsilon \ \
\forall \ t \in (0,\bar{t}\,].
\end{equation}
Then, taking a smaller $\varepsilon$ if necessary,
\begin{equation*}
I((v_{n})_{\bar{t}})+I((w_{n})_{\bar{t}})<0.
\end{equation*}
Then $I((v_{n})_{\bar{t}})<0$ or $ I((w_{n})_{\bar{t}})<0$. In any
case Lemma \ref{lemma11} implies that $t^{v_n}<\bar{t}$ (recall
that we are assuming $t^{v_n} \leq t^{w_n}$).

For the lower bound, take $\tilde{t}=\Big(\frac{m}{B}\Big)^{1/N}$,
where $B$ is chosen as in \eqref{B}. Let us point out that
$\tilde{t}<1$. For any $t\leq \tilde{t}$,
\begin{eqnarray*}
 I((u_n)_t) &\leq& \frac{\tilde{t}^{N}}{2}
\Big ( \int_{\R^{N}}|\nabla u_{n}| ^{2}+\int_{\R^{N}}V(\bar{t}x)
u_{n}^{2} +\int_{\R^{N}}|\nabla u_{n}|^{2} u_{n}^{2} \Big)\\
& \leq& \frac{m}{2B} \int_{\R^{N}}\left(|\nabla u_{n}|
^{2}+V_{\infty} u_{n}^{2} +|\nabla u_{n}|^{2} u_{n}^{2}\right) \le
\frac{m}{2}.
\end{eqnarray*}
By \eqref{pos} and \eqref{550},
\begin{equation}\label{greater}
I((u_n)_{t^{v_n}}) \geq I((v_{n})_{t^{v_n}})+I((w_{n})_{t^{v_n}})-
C \varepsilon \geq m - C \varepsilon
\end{equation}
and the right hand side can be made greater then
 $ m/2$, by choosing a small $\e$. We conclude that
$t^{v_n} > \tilde{t}$ and the claim is proved.

%



\bigskip

Since $u_n \in M$, $f_u$ reaches its maximum at $t=1$. Then,
$$m
\leftarrow I(u_n) \geq I((u_n)_{t^{v_n}}) 
$$
and using  \eqref{greater} we deduce, for $n$ large,
$I((w_{n})_t)\leq 2 C \e$ for all $ t \in (0, t^{v_n})$. Moreover,
for any $t \in (0,\tilde{t})$:
\begin{eqnarray*}
2 C \varepsilon &\ge& I((w_{n})_{t}) \\
&\ge& \frac{t^{N+2}}{2}\left[\int_{\R^{N}}|\nabla
w_{n}|^{2}+\int_{\R^{N}}V(tx)w_{n}^{2} +\int_{\R^{N}}|\nabla
w_{n}|^{2}w_{n}^{2}\right]-\frac{t^{N+p+1}}{p+1}\int_{\R^{N}}|w_{n}|^{p+1}\\
&\ge&  \frac{t^{N+2}}{2} q_n -D t^{N+p+1}
\end{eqnarray*}
where $$q_n= \int_{\R^{N}}|\nabla w_{n}|^{2}+ V_0 \int_{\R^{N}}
w_{n}^{2} +\int_{\R^{N}}|\nabla w_{n}|^{2}w_{n}^{2}$$ is bounded
(independently of $\e$) and $D>A$. Observe that $
\frac{t^{N+2}}{2} q_n -D t^{N+p+1} = \frac{t^{N+2}}4 q_n$ for $t =
\Big(\frac{q_n}{4D}\Big)^{\frac{1}{p-1}}$. By taking a larger $D$
we can assume that $\Big(\frac{q_n}{4D}\Big)^{\frac{1}{p-1}} \leq
\tilde{t}$. With this choice of $t$, we obtain:
$$ 2 C \varepsilon \geq I((w_{n})_{t}) \geq
\Big(\frac{q_n}{4D}\Big)^{\frac{N+2}{p-1}}\frac{q_{n}}{4D} \geq c
q_n^{\frac{N+p+1}{p-1}}.$$ In other words,
\begin{equation} \label{estimatew} \|w_n\|_{H^1} + \|w_n^2 \|_{H^1} \leq C
\e^{\frac{p-1}{2(N+p+1)}} \mbox{ for some }C>0 \mbox{ independent
of }\e.
\end{equation}

\bigskip

In the case  $t^{v_{n}} >  t^{w_{n}}$, we can argue analogously to
conclude that $\|v_n\|_{H^1} + \|v_n^2 \|_{H^1} \leq C
\e^{\frac{p-1}{2(N+p+1)}}$ for some $C>0$. But, choosing small
$\e$, this contradicts \eqref{nonvanish}, so \eqref{estimatew}
holds. This finishes the proof of Step 3.

\bigskip

We are now in conditions to conclude the proof of Theorem
\ref{teo1}.

\bigskip
{\bf Step 4:} The infimum of $I|_M$ is achieved.
\bigskip

 Recall, see Step 2, that $z_n =u_{n}(\cdot+x_{n})$, $z_n \rightharpoonup
z$ and $z_n^2 \rightharpoonup z^2$ (both weak convergences are in
$H^1(\R^{N})$). Moreover, by compactness, we have that $z_n\to z$
in
$L^2_{loc}(\R^N)$. 
 Finally, observe that $z\neq
0$ since, by \eqref{nonvanish},
$$ \delta < \liminf_{n\rightarrow+\infty} \intr |v_n|^{p+1} \leq  \lim_{n\rightarrow+\infty}
\int_{B_{0}(2R) }|z_n|^{p+1} = \int_{B_{0}(2R)} |z|^{p+1}.$$

Recall also that $u_n=v_n+w_n$, with $\|w_n\|_{H^1} +
\|w_n^2\|_{H^1} \leq C \e^{\frac{p-1}{2(N+p+1)}}$. In the
following estimate we use H{\"o}lder inequality to get:
\begin{eqnarray} \label{casi1}
 \dis \int_{\R^{N}} \left| u_n^2  - v_n^2 \right | &\leq & \intr |w_n| (|u_n| + |v_n|) \\
 &\leq&
 \dis \Big( \intr w_n^2 \Big)^{1/2} \Big( \intr(|u_n| + |v_n|)^{2} \Big)^{1/2} \leq C \e^{\frac{p-1}{2(N+p+1)}}\nonumber\,.
\end{eqnarray}

On the other hand,
$$ \int_{\R^{N}} v_n^2 \leq
\int_{B(0,2R)} z_n^2 \to \int_{B(0,2R)} z^2 \leq \intr z^2.$$

Combining this estimate with \eqref{casi1}, we obtain that
$$ \liminf_{n\rightarrow+\infty} \intr z_n^2 =
\liminf_{n\rightarrow+\infty} \intr u_n^2 \leq \intr z^2 + C
\e^{\frac{p-1}{2(N+p+1)}}. $$ Since $\e$ is arbitrary, we get that
$z_n \to z$ in $L^2(\R^N)$ and, by interpolation, $z_n \to z$ in
$L^q(\R^N)$ for all $q \in [2, \frac{4N}{N-2})$.

We discuss two cases:

\bigskip

{\bf Case 1:} $\{x_n\}$ is bounded. Assume, passing to a
subsequence, that $x_n \to x_0$. In this case $u_n \rightharpoonup
u$, $u_n^2 \rightharpoonup u^2$ (both weak convergences are in
$H^1$), $u_n \to u$ strongly in $L^q(\R^N)$ for any $q \in
[2,\frac{4N}{N-2})$, where $u= z(\cdot - x_0)$.

In the following we just need to recall the expression of
$I((u_n)_t)$, see \eqref{f}:
$$ m = \lim_{n\rightarrow+\infty} I(u_n) \geq \liminf_{n\rightarrow+\infty} I((u_n)_t) \geq I(u_t) \ \
\forall \ t>0.$$ So, $\max_t I(u_t) = m$ and $u_n \to u$, $u_n^2
\to u^2$ (both convergences are now strong in $H^1(\R^{N})$). In
particular, $u\in M$ is a minimizer of $I|_M$.

\bigskip

{\bf Case 2: } $\{x_n\}$ is unbounded. In this case, by Lebesgue
Convergence Theorem and condition (V1):
\begin{eqnarray*}
 \lim_{n\rightarrow+\infty} \intr V(tx)u_n^2(x) \, dx &=& \lim_{n\rightarrow+\infty} \intr V(t(x+x_n)) z_n^2(x) \, dx \\
 &=&V_{\infty} \intr  z^2
\geq \intr V(tx)z^2(x) \,dx= \lim_{n\rightarrow+\infty} \intr
V(tx)z_n^2(x) \,dx
\end{eqnarray*}
for any $t>0$ fixed. Therefore,
$$ m= \lim_{n\rightarrow+\infty} I(u_n) \geq \liminf_{n\rightarrow+\infty} I((u_n)_t) \geq \liminf_{n\rightarrow+\infty}
I((z_n)_t) \geq I(z_t) \ \ \ \forall \, t>0.$$ So, taking $t^{z}$
so that $f_z(t)=I(z_t)$ reaches its maximum, we get that
$z_{t^{z}}\in M$ and is a minimizer for $I|_M$.

\bigskip

Having a minimum of $I|_{M}$, the fact that it is indeed a
solution of our equation, is based on a general idea used in \cite{LWW2}.

\medskip

\subsection{Proof  of Theorem \ref{teo2}}
Let $u\in M$ be a minimizer of the functional $I|_M$. Recall that,
by Lemma \ref{lemma11}, $I(u)=\inf_{v\in X, \, v\ne0}\max_{t>
0}I(v_{t})=m$.

We argue  by contradiction by assuming that $u$ is not a weak
solution of \eqref{Eq}. In such a case, we can choose $\phi\in
C^{\infty}_{0}(\R^{N})$ such that
\begin{equation*}
\langle I'(u),\phi \rangle=  \int_{\R^{N}}\nabla u
\nabla\phi+\int_{\R^{N}}V(x)u\phi+\int_{\R^{N}}\nabla(u^{2})\nabla(u
\phi) -\int_{\R^{N}}|u|^{p-1}u\phi <-1.
\end{equation*}
Then we fix $\varepsilon>0$  sufficiently small such that
\begin{equation*}
\langle I'(u_{t}+\sigma \phi),\phi\rangle\le -\frac{1}{2},\ \ \
\forall\, |t-1|, \,|\sigma|\le\varepsilon
\end{equation*}
and  introduce a cut-off function $0\le\eta\le1$ such that $\eta
(t)=1$ for $|t-1|\le\varepsilon/2$ and $\eta(t)=0$ for
$|t-1|\ge\varepsilon.$

We perturb the  original curve $u_{t}$ by defining, for $t\ge0$
\begin{equation*}
\gamma (t)= \left\{ \begin{array}{ll}
u_{t} & \textrm{if } |t-1|\ge\varepsilon \\
u_{t}+\varepsilon\eta(t)\phi& \textrm{if } |t-1|<\varepsilon\,.\\
\end{array} \right.
\end{equation*}
Note that $\g(t)$ is a continuous curve in the metric space
$(X,d)$ and, eventually  choosing a smaller $\varepsilon$, we
obtain that $d_{X}(\g(t),0)>0$ for $|t-1|<\varepsilon$.


\medskip {\bf Claim:} $\sup_{t\ge0}I(\g(t))<m.$

\medskip

Indeed, if $|t-1|\ge\varepsilon,$ then $I(\g(t))=I(u_{t})<I(u)=m.$
If $|t-1|<\varepsilon$, by using the mean value theorem to the
$C^{1}$ map $[0,\varepsilon] \ni \sigma  \mapsto
I(u_{t}+\sigma\eta(t)\phi)\in \R$, we find, for a suitable
$\bar\sigma\in (0,\varepsilon)$,
\begin{eqnarray*}
I(u_{t}+\varepsilon\eta(t)\phi)&=&I(u_{t})+\langle I'(u_{t}+\bar\sigma\eta(t)\phi ), \eta(t)\phi\rangle \nonumber\\
&\le &I(u_{t})-\frac{1}{2} \eta(t) \\
&<&m
\end{eqnarray*}

\medskip

To conclude observe that $J(\gamma(1-\varepsilon))>0$ and
$J((\gamma(1+\varepsilon))<0 $. By the continuity of the map
$t\mapsto J(\g(t))$ there exists
$t_{0}\in(1-\varepsilon,1+\varepsilon)$ such that $J(\gamma(t_{0}))=0$. Namely, 
$\gamma(t_{0})=u_{t_{0}}+\varepsilon\eta(t_{0})\phi\in M$ and
$I(\gamma(t_0))<m$; this gives the desired contradiction.

\medskip So far we have proved that the minimizer of $I|_M$ is a
solution. Since any solution of \eqref{Eq} belongs to $M$ (see
Section 2), the minimizer is a ground state.

Moreover, consider $u \in M$ a minimizer for $I|_M$. Then, the
absolute value $|u|\in M$ is also a minimizer, and hence a
solution. By the classical maximum principle (recall that
solutions are $C^2$), $|u|>0$.


%

\bigskip

{\bf Acknowledgement:} The authors would like to thank Tommaso
Leonori and Marco Squassina for several discussions on the
subject.

\end{document}